\documentclass[a4,12pt,twosided,reqno]{amsart}

\usepackage{amscd,amssymb}
\usepackage[arrow,matrix]{xy}
\usepackage{graphicx}
\usepackage{color}

\oddsidemargin=0.3in \evensidemargin=0.3in \theoremstyle{plain}

\newtheorem{theorem}{Theorem}[section]
\newtheorem{lemma}{Lemma}[section]

\newtheorem{cor}{Corollary}[section]
\newtheorem{defn}{Definition}[section]
\newtheorem{ex}{Example}[section]
\theoremstyle{remark}
\newtheorem{rem}{Remark}[section]

\begin{document}

\begin{center}
{\Large {\bf Some Unique Fixed Point Theorems in
 Multiplicative Metric Space}}
\end{center}
\vskip.3cm
\begin{center}
{\bf {Muhammad Sarwar and Badshah-e-Rome }}
\end{center}
\vskip+0.5cm

{\bf Abstract.}
\"{O}zavsar and Cevikel( Fixed point of multiplicative contraction mappings on multiplicative metric space.arXiv:1205.5131v1 [math.GN] 23 may 2012)initiated the concept of the multiplicative metric space in such a way  that the usual  triangular inequality is replaced by ``multiplicative triangle inequality  $d(x,y)\leq d(x,z).d(z,y)$ for all $x,y,z\in X$".
In this manuscript, we discussed  some unique fixed point theorems  in the context of  multiplicative metric spaces. The established results carry  some well known results from the literature to multiplicative metric space.
 we note that  some fixed point theorems can be deduced in multiplicative metric space by using the established results.
Appropriate examples are also given.

\vskip+0.2cm
{\bf 2000 Mathematics Subject Classification:} 47H10, 54H25, 55M20.
\vskip+0.2cm
{\bf Key Words and Phrases}: Complete multiplicative metric space, multiplicative contraction mapping, multiplicative Cauchy sequence, fixed point.

\section{Introduction}
 Banach contraction principle has been a
 very advantageous and efficacious means in nonlinear analysis. Various authors have generalized Banach contraction principle in different spaces. Singxi \textit{et al}. \cite{B2} and Sastry \textit{et al.} \cite{B3} studied some common fixed point theorems for different mappings on a 2-metric space. Dhage \cite{B5} proved fixed point results in D-metric space. Branciari \cite{1} introduced the concept of generalized metric spaces by putting a  general inequality condition in place of usual triangular inequality in metric space.  Gu \textit{et al}.\cite{B1} derived  some common
fixed point theorems related to weak commutative mappings on a complete metric space. Furthermore  Mustafa and Sims \cite{B7,B8} studied various results on the class of
generalized metric spaces.  Agarwal \textit{et al.} \cite{B4} discussed some fixed point results related to monotone operators in the setting of metric space equipped with a partial order using a weak generalized contraction-type mapping. Suzuki \cite{B6} characterized metric completeness by  generalized
 Banach contraction principle. Aydi \textit{et al}. \cite{2} obtained  fixed point results for weak contractive
mappings in complete Hausdorff generalized metric spaces.

\"{O}zavsar and Cevikel\cite{A} introduced the concept of multiplicative contraction mappings and proved some fixed point theorems of such mappings on a complete multiplicative metric space. They also gave some topological properties of the relevant multiplicative metric space. Hxiaoju \textit{et al}.\cite{He} studied  common fixed points for weak commutative mappings on a multiplicative metric
space.
 For further details about multiplicative metric
 space and related concepts, we refer the reader to\cite{mutip, A}. Fixed-point theory and self mappings satisfying certain contraction conditions has many applications and has been an important area of various research activities \cite{B4,C,D,E,F}.

 In the current article we studied  some unique  fixed point theorems  in the setting of multiplicative metric space. The derived results carry  fixed point results  of  Zamfirescu \cite{A1}  and  Aydi \textit{et al} \cite{2}   to multiplicative metric spaces.
\section{Prelimenaries}
\begin{defn}\cite {mutip}  Multiplicative metric on a nonempty set $X$ is a mapping \\$ d: X\times\ X \to R $ satisfying the following condition:

 $(1)\  \ d(x,y)\geq 1 \text{  for all } x,y\in X;$

 $ (2)\  \ d(x,y) = 1 \ \ \  \hbox{if and only if } x=y;$

 $(3)\  \ d(x,y) = d(y,x)\text{ for all } x,y\in X;$

$(4)\  \ d(x,z)\leq d(x,y).d(y,z) \text{ for all } x,y,z\in X.$ And the pair $(X, d)$ is called multiplicative metric space.
\end{defn}

\begin{ex}
 Let $R^n_+$  be the collection of all $n-$tupples of positive real numbers. And let \\$ d^*: R^n_+\times R^n_+\to R$  be defined as\\
 \begin{equation*}d^*(x,y)=|\frac{x_1}{y_1}|^*\cdot |\frac{x_2}{y_2}|^*\cdots |\frac{x_n}{y_n}|^*
   \end{equation*} where $x= (x_1,x_2\cdots x_n),y= (y_1,y_2\cdots y_n)\in R^n_+ \text{  and  } |\cdot|^*: R_+\to R_+ $ is defined as follows \begin{equation*}|a|^*=\begin{cases} a & \text{if $a\geq 1$}\\ \frac{1}{a} & \text{if $a<1$}\end{cases}
 \end{equation*}Then clearly $d^*(x,y)$ is a multiplicative metric introduced in \cite{mutip}.
\end{ex}

\begin{ex}
Let $(X,d)$ be a metric space, then the mapping $d_a$ defined on $X$ as follows is multiplicative metric, $d_a(x,y)=a^{d(x,y)}$ where $a>1$ is a real number. For discrete metric $d$ the corresponding mapping $d_a$ called discrete multiplicative metric is defined as:

$d_a(x,y)=a^{d(x,y)}=\begin{cases} 1 & \text{if $x=y$}\\a & \text{if $x\neq y$}\end{cases}$
\end{ex}

\begin{ex}
Let $d: R\times R\to [1,\infty)$ be defined as $d(x,y)= a^{|x-y|}$ where $x,y\in R$ and $a>1$. Then $d(x,y)$ is multiplicative metric.
\end{ex}


\begin{rem}
  Neither every metric is multiplicative metric nor every multiplicative metric is metric. The mapping $ d^*$ defined above is multiplicative metric but not metric as it doesn't satisfy triangular inequality. Consider
  $ d^*(\frac{1}{3},\frac{1}{2})+d^*(\frac{1}{2},3)=\frac{3}{2}+6=7.5<9=d^*(\frac{1}{3},3)$. On the other hand the usual metric on $R$ is not multiplicative metric as it doesn't satisfy multiplicative triangular inequality. As $d(2,3)\cdot d(3,6)=3<4= d(2,6)$
\end{rem}

\begin{defn}\cite{A} Let $(X, d)$ be a multiplicative metric space. Then we have the following inequality
 \begin{equation*}|\frac{d(x,z)}{d(y,z)}|^*\leq d(x,y)\ \  \  \hbox{for all}\ \ x, y,z\in X,
 \end{equation*}
 which  is called multiplicative reverse triangular inequality.
\end{defn}
\begin{defn}\cite{A} Let $(X,d)$ be a multiplicative metric space. If $ a\in X $ and $ r >1  $ then  a subset
$$B _r(a)=B(a;r)=\{ x\in X: d(a,x)<r\}$$
 of $ X $ is called multiplicative open ball centered at $ a $  with radius $ r $. Analagusly one can defined multilicative closed ball as
 $$\bar{B} _r(a)=\bar{B}(a;r)=\{ x\in X: d(a,x)\leq r\}.$$
  \end{defn}
  \begin{defn}\cite{A} Let $ A $ be any subset of a multiplicative metric space $(X,d)$. A point $x\in X$ is called limit point of $ A $
  if and only if $ (A \cap B _\epsilon(x))-\{x\} \neq \phi$ for every $\epsilon > 1$
  \end{defn}

 \begin{defn}\cite{A}
 Let $(X,d) \text{ and } (Y,\rho) $ be given multiplicative metric spaces and $a\in X$. A function $f:(X,d)\to (Y,\rho) $ is said to multiplicative continuous at $a$, if for given $\epsilon > 1$, there exists a $\delta> 1$ such that $ d(x,a)< \delta \Rightarrow d(f(x),f(a))<\epsilon $ or equivalently $f(B(a;\delta))\subset B(f(a);\epsilon)$. Where $B(a;\delta) \text{ and }B(f(a);\epsilon) $ are open balls in  $(X,d) \text { and } (Y,\rho)$ respectively. The function $ f$ is said to be continuous on $X $ if it is continuous at each point of $X$.
 \end{defn}

 \begin{defn}\cite{A}
  A sequence $\{x_n\}$ in a multiplicative metric space $(X,d)$ is said to be  multiplicative convergent to a point $x\in X$ if for a given $\epsilon>1$ there exits a positive integer $n_0$ such that  $d(x_n,x) < \epsilon \text{ for all } n\geq n_0 $ or equivalently, if for every multiplicative open ball $B_\epsilon(x)$ there exists a positive integer $n_0$ such that $n\geq n_0 \Rightarrow x_n \in B_\epsilon(x)  $ then the sequence $\{x_n\}$ is said to be  multiplicative convergent to a point $x\in X$ denoted by $x_n \rightarrow x (n \rightarrow \infty).$
 \end{defn}

  \begin{defn}\cite{A}  A sequence $\{x_n\}$ in a multiplicative metric space $(X,d)$ is said to be  multiplicative Cauchy sequence if for all   $\epsilon>1$ there exits a positive integer $n_0$ such that $$d(x_n,x_m) < \epsilon \text{ for all } n,m\geq n_0. $$
  \end{defn}

  \begin{defn}\cite{A} A multiplicative metric space $(X,d)$ is said to be complete if every multiplicative Cauchy sequence in $X \text{converges in } X. $
  \end{defn}

  \begin{defn}\cite{A} Let $(X, d)$ be a multiplicative metric space. A mapping $ f : X \to X $ is called multiplicative contraction if there exists a real constant $\lambda\in [0, 1)$ such that
$$d(f(x_1), f(x_2)) \leq d(x_{1}, x_{2})^\lambda \text{ for all } x, y \in X$$.
\end{defn}

\begin{theorem}\cite{A} In a multiplicative metric space every multiplicative convergent sequence is multiplicative Cauchy sequence. \end{theorem}

  \begin{lemma}\cite{A} A sequence $\{x_n\}$ in a multiplicative metric space $(X,d)$ is multiplicative Cauchy sequence if and only if $ d(x_n,x_m)\rightarrow 1 (n,m \rightarrow \infty)$
   \end{lemma}

\section{Main results}
In this section we  prove  some unique fixed point theorems  in the setting of  multiplicative metric space.
\begin{theorem}\label{t2}
Let $(M,d)$ be a complete multiplicative metric space and let $\xi ,\eta ,\lambda $ be real numbers with $\xi\in [0,1)$ and $\eta ,\lambda\in[0,\frac{1}{2}) $ and  $T : M \rightarrow M $ a function such that for each pair of different points $x , y \in M$ at least one of the following conditions is satisfied:

$ (1)\  \ d(Tx,Ty)\leq (d(x,y))^\xi$

$(2)\  \ d(Tx,Ty)\leq \big(d(x,Tx).d(y,Ty)\big)^\eta$

$(3)\  \ d(Tx,Ty)\leq \big(d(x,Ty).d(y,Tx)\big)^\lambda$\\
Then T has a unique fixed point.
\end{theorem}

\begin{proof}
Consider the number $\delta=max \{\xi, \frac{\eta}{1-\eta},\frac{\lambda}{1-\lambda}\}$. Obviously, $\delta<1.$ Now choose $x_0\in M$ arbitrarily and fix an integer $n \geq 0.$Take $x= T^nx_0$ and $y=T^{n+1}x_0$.
\\Let $x\neq y$, otherwise $x$ is a fixed point of $T$. If for these two points condition (1) is satisfied then $$d(T^{n+1}x_0,T^{n+2}x_0)\leq (d(T^nx_0,T^{n+1}x_0))^\delta.$$
If for $x$ and $y$ condition (2) is satisfied then
 \begin{eqnarray*}d(T^{n+1}x_0,T^{n+2}x_0)&\leq& \big(d(T^nx_0,T^{n+1}x_0).d(T^{n+1}x_0,T^{n+2}x_0)\big)^\eta\\
&=& \big (d(T^nx_0,T^{n+1}x_0)\big)^\eta.\big(d(T^{n+1}x_0,T^{n+2}x_0)\big)^\eta.
\end{eqnarray*}
Which implies that
$$ \big(d(T^{n+1}x_0,T^{n+2}x_0)\big)^{1-\eta}= \{d(T^nx_0,T^{n+1}x_0)\}^\eta.$$
Finally we have
$$ d(T^{n+1}x_0,T^{n+2}x_0)
=\big(d(T^nx_0,T^{n+1}x_0)\big)^{\frac{\eta}{1-\eta}}\leq \big(d(T^nx_0,T^{n+1}x_0)\big)^\delta.$$
 If for $x$ and $y$ condition (3) is satisfied then\\
\begin{multline*}
d(T^{n+1}x_0,T^{n+2}x_0)\leq \{d(T^nx_0,T^{n+2}x_0).d(T^{n+1}x_0,T^{n+1}x_0)\}^\lambda\\
= \big(d(T^nx_0,T^{n+2}x_0)\big)^\lambda.1\leq\big(d(T^nx_0,T^{n+1}x_0).d(T^{n+1}x_0,T^{n+2}x_0)\big)^\lambda.  
\end{multline*}
Hence one can get
$$d(T^{n+1}x_0,T^{n+2}x_0)
=\big(d(T^nx_0,T^{n+1}x_0)\big)^{\frac{\lambda}{1-\lambda}}\leq \big(d(T^nx_0,T^{n+1}x_0)\big)^\delta.$$
This inequality is true for every $n,$  therefore continuing the same  procedure  we have
\begin{equation*}d(T^{n+1}x_0,T^{n+2}x_0)\leq ((d(T^nx_0,T^{n+1}x_0))^\delta \leq (d(T^{n-1}x_0,T^nx_0)^{\delta^2}\cdots \leq (d(x_0,Tx_0)^{\delta^{n+1}}.
\end{equation*}
 Let $m,n\in N $ such that $ m>n$,  then using multiplicative triangular inequality we get
  \begin{eqnarray*} d(T^mx_0,T^nx_0)&\leq& d(T^mx_0,T^{m-1}x_0).d(T^{m-1}x_0,T^{m-2}x_0)\cdots d(T^{n+1}x_0,T^nx_0)\\&\leq& (d(Tx_0,x_0))^{\delta^{m-1}+\delta^{m-2}+\delta^{m-3}\cdots +\delta^n}\\&=& (d(Tx_0,x_0))^\frac{\delta^n-\delta^m}{1-\delta}.
  \end{eqnarray*}
  Since $\delta<1$, therefore
  \begin{equation*}\lim_{m,n\to\infty}d(T^mx_0,T^nx_0)=1.
   \end{equation*}
   Which shoes that  $\{ T^nx_0\} $ is multiplicative Cauchy sequence, and due the completeness of M, it converges to some $z\in M$.

 We now claim that $T(z)=z$. Suppose by way of contradiction that $T(z)\neq z $, then $d(Tz,z)=\epsilon> 1$ and consider the ball
 $$B = \{x\in M: d(x,z)< \epsilon^{\frac{1}{4}} \}$$.
  Now for any $x\in B$ we have
 $d(z,Tz)\leq d(z,x).d(x,Tz).$
  Which implies that $d(x,Tz)> \epsilon^\frac{3}{4}$. As $\{ T^nx_0\}$ converges to $z$ so every ball with center $z$ contains all but finite number of terms of $\{ T^nx_0\}$, hence there must be a natural number $p$ such that $ T^nx_0\in B \   \ \forall n\geq p .$\\
   According to the hypothesis of the theorem for the two points $T^px_0$ and $z$ at least one of the following  three conditions must be satisfied:

 $(1^{\prime})$ \  \  \ \  \ $d(T^{p+1}x_0,Tz)\leq (d(T^px_0,z))^\xi$

 $ (2^{\prime})$  \  \  \ \  \ $  d(T^{p+1}x_0,Tz)\leq \big(d(T^px_0,T^{p+1}x_0).d(z,Tz)\big)^\eta$

  $(3^{\prime})$\  \  \ \  \ \ $ d(T^{p+1}x_0,Tz)\leq \big(d(T^px_0,Tz) \cdot d(T^{p+1}x_0 ,z)\big)^\lambda$

 condition $(1^{\prime})$ is not possible because
 $$(d(T^px_0,z))^\xi<d(T^px_0,z)<\epsilon^\frac{1}{4}<\epsilon^\frac{3}{4}<d(T^{p+1}x_0,Tz)\hspace{1.5cm}\because T^{p+1}x_0\in B$$

 and
 \begin{multline*}
 \big(d(T^px_0,T^{p+1}x_0).d(z,Tz)\big)^\eta<\big(d(T^px_0,T^{p+1}x_0).d(z,Tz)\big)^\frac{1}{2}\\
 \leq
 \big(d(T^px_0,z).d(z,T^{p+1}x_0).d(z,Tz)\big)^\frac{1}{2}
<(\epsilon^\frac{1}{4}.\epsilon^\frac{1}{4}.\epsilon)^\frac{1}{2}=\epsilon^\frac{3}{4}<d(T^{p+1}x_0,Tz)
\end{multline*}
contradits condition $(2^{\prime}).$

Similarly  the following lines  contradicting condition $(3^{\prime})$
 \begin{multline*} \big(d(T^px_0,Tz)\cdot d(T^{p+1}x_0 ,z)\big)^\lambda<\big(d(T^px_0,Tz)\cdot d(T^{p+1}x_0 ,z)\big)^\frac{1}{2}\\
 \leq\big(d(T^px_0,z)\cdot d(z,Tz) \cdot d(T^{p+1}x_0 ,z)\big)^\frac{1}{2}
 <(\epsilon^\frac{1}{4}\cdot\epsilon\cdot\epsilon^\frac{1}{4})^\frac{1}{2}=\epsilon^\frac{3}{4}<d(T^{p+1}x_0,Tz)
 \end{multline*}
 Therefore  none of the three conditions is satisfied for $T^px_0$ and $z$.

 Hence $T(z)=z$.

 \textbf{Uniqueness:} Suppose that $z^\prime\neq z $ is another fixed point of $T$ for some $z^\prime\in M $, i.e  $T(z^\prime)=z^\prime$. Since $d$ is multiplicative metric, therefore $d(z, z^{\prime})>1.$

 For $z, z^{\prime}\in M$ Clearly we have
  $$d(Tz,Tz^\prime)>d(z,z^\prime)^\xi \   \ \mbox{  for all } \xi\in [0,1)\ ,$$
  $$ d(Tz, Tz^{\prime})> d(z,Tz).d(z^\prime,Tz^\prime)^{\eta} \ \ \text{ for all } \eta\in[0,\frac{1}{2})$$
  and
  \begin{multline*}
   d(Tz,Tz^\prime)= d(Tz,Tz^\prime)^\frac{1}{2}.d(Tz,Tz^\prime)^\frac{1}{2}\\=\{d(Tz,z^\prime).d(z,Tz^\prime\}^\frac{1}{2}>\{d(Tz,z^\prime).d(z,Tz^\prime)\}^\lambda \    \  \text{ for all } \lambda\in[0,\frac{1}{2}).
  \end{multline*}
Thus none of the three conditions of the theorem is satisfied for the points  $z$ and  $ z^\prime$ of $M$. Hence fixed point of $T$ is unique.
 \end{proof}
 Theorem \ref{t2} yields the following corollaries.

\begin{cor}\label{c1} Let $M$ be a complete multiplicative metric space, $\xi\in [0,1) \text{ and let } T: M\to M $ be a function such that for each pair of distinct points of $M $ condition (1) of Theorem \ref{t2} is satisfied, Then $T$ has a unique fixed point.\end{cor}
\begin{cor}\label{c2}Let $M$ be a complete multiplicative metric space, $\eta\in[0,\frac{1}{2})\text{ and let } T: M\to M $ be a function such that for each couple of distinct points of $M $ condition (2) of Theorem \ref{t2} is satisfied, Then $T$ has a unique fixed point.\end{cor}
\begin{cor}\label{c3} Let $M$ be a complete multiplicative metric space, $\lambda\in[0,\frac{1}{2})\text{ and let } T: M\to M $ be a function such that for each couple of distinct points of $M $ condition (3) of Theorem \ref{t2} is satisfied, Then $T$ has a unique fixed point. \end{cor}
\begin{rem}
  Corollaries \ref{c1}, \ref{c2} and \ref{c3} are the results of "{O}zavsar and Cevikel\cite{A}.
\end{rem}

The following result generalizes Theorem \ref{t2}.

\begin{theorem}\label{t3'}
Let $(M,d)$ be a  multiplicative metric space and let $\xi ,\eta ,\lambda $ be real numbers with $\xi\in [0,1)$ and $\eta ,\lambda\in[0,\frac{1}{2}) $ and  $T : M \rightarrow M $ a function such that for each pair of different points $x , y \in M$ at least one of the condition $(1)$, $(2)$, $(3)$ of Theorem\ref{t2} is satisfied. If for some $x_0$ the sequence $\{T^{n}(x_0)\}$ has a limit point $z_1$ in $M$, then $z_1$ is the unique fixed point of $T$.
\end{theorem}
\begin{proof}
To prove the required result, following the lines in the the proof of Theorem \ref{t2} we conclude that $\{T^{n}(x_0)\}$ is  multiplicative  Cauchy sequence. Since $T^{n}(x_0)\rightarrow z_1$. Again following the proof of  Theorem \ref{t2} we get that $z_1$  is fixed point of $T.$
\end{proof}

Like  Theorem \ref{t2} one can deduce corollaries from  Theorem \ref{t3'}.

\begin{theorem}\label{t22}
Let $S$ be a set in complete multiplicative metric space $(M,d)$, $\xi ,\eta ,\lambda $ be real numbers with $\xi\in [0,1)$ and $\eta ,\lambda\in[0,\frac{1}{2}) $ and  $T : S \rightarrow S $ a function such that for each pair of different points in $S$, at least one of the three conditions of Theorem \ref{t2} is satisfied. Then, for $x_0 \in S$, the sequence $\{T^n(x_0)\}^\infty_{n=0}$ converges to a point in $M$, independent on the choice of $x_0$.
\end{theorem}
\begin{proof}
Like in the proof of Theorem \ref{t2}, it can be shown that there exists a point $z\in M$ such that
\begin{equation*}\lim_{n\to\infty} T^n(x_0)=z.\end{equation*}

Let $y_0$ be an arbitrary point in $S$ such that $T^n(y_0)\to \acute{z},$ where $\acute{z}\in M$. Now we have to prove that $z=\acute{z}.$
Suppose on the contrary that  $z\neq\acute{z}$, then $d(z,\acute{z})>1,$ and choose the real number $\epsilon>1$ such that
\begin{equation*}
1<\epsilon< d(z,\acute{z})^{\frac{1}{2}\cdot min\{\frac{1-\xi}{1+\xi},\frac{1}{2},\frac{1-2\lambda}{1+2\lambda}\}}.
 \end{equation*}
As $\{T^n(x_0)\}$ and $\{T^n(y_0)\}$ are convergent, so there will be a natural number $N$ such that $d(T^n(x_0),z)\leq \epsilon$ and $d(T^n(y_0),\acute{z})\leq \epsilon$ for every $n\geq N.$\\
Let $x=T^n(x_0)$ and $y=T^n(y_0)$ for $n\geq N$.  Clearly $x\neq y$. Suppose for these $x$ and $y$ condition (1)of Theorem \ref{t2} is satisfied. Then,
\begin{equation*}
d(T^{n+1}(x_0),T^{n+1}(y_0))\leq d(T^n(x_0),T^n(y_0))^\xi.
\end{equation*}
But
\begin{eqnarray*}
d(T^n(x_0),T^n(y_0))^\xi&\leq& \big(d(T^n(x_0),z)\cdot d(z,\acute{z})\cdot d(\acute{z},T^n(y_0))\big)^\xi
\leq d(z,\acute{z})^\xi\cdot\epsilon^{2\xi}.
\end{eqnarray*}
Since
\begin{equation*}\epsilon< d(z,\acute{z})^{\frac{1}{2}\cdot{\frac{1-\xi}{1+\xi}}}
\end{equation*}
which implies that
\begin{equation*} d(z,\acute{z})^\xi\cdot\epsilon^{2\xi}< \frac{d(z,\acute{z})}{\epsilon^2}.
\end{equation*}
Therefore
\begin{eqnarray*}
d(T^n(x_0),T^n(y_0))^\xi< \frac{d(z,\acute{z})}{\epsilon^2}
\leq\frac{d(z,\acute{z})}{d(T^{n+1}(x_0),z)\cdot d(\acute{z},T^{n+1}(y_0))}\leq d(T^{n+1}(x_0),T^{n+1}(y_0)),
\end{eqnarray*}
which gives a contradiction.

Again if for the above  $x$ and $y$ condition (2)of  Theorem \ref{t2} is satisfied. Then
\begin{equation*}
d(T^{n+1}(x_0),T^{n+1}(y_0))\leq (d(T^n(x_0),T^{n+1}(x_0))\cdot d(T^n(y_0),T^{n+1}(y_0)))^\eta .
\end{equation*}
Now,
\begin{multline*}
(d(T^n(x_0),T^{n+1}(x_0))\cdot d(T^n(y_0),T^{n+1}(y_0)))^\eta \\
\leq (d(T^n(x_0),z)\cdot d(z,T^{n+1}(x_0))\cdot d(T^n(y_0),\acute{z})\cdot d(\acute{z},T^{n+1}(y_0)))^\eta\leq \epsilon^{4\eta}
\end{multline*}
Since
\begin{equation*}\epsilon< d(z,\acute{z})^{\frac{1}{2}\cdot \frac{1}{2}}< d(z,\acute{z})^{\frac{1}{2}\cdot{\frac{1}{1+2\eta}}}.
\end{equation*}
Hence
\begin{equation*}
\epsilon^{4\eta}<\frac{d(z,\acute{z})}{\epsilon^2}.
\end{equation*}
Therefore
\begin{eqnarray*}
(d(T^n(x_0),T^{n+1}(x_0))\cdot d(T^n(y_0),T^{n+1}(y_0)))^\eta\leq \epsilon^{4\eta}<\frac{d(z,\acute{z})}{\epsilon^2}\leq d(T^{n+1}(x_0),T^{n+1}(y_0)),
\end{eqnarray*}
which is a contradiction.

Finally suppose for the same point  $x=T^{n}(x_0)$ and $y=T^{n}(y_0)$ condition (3) of Theorem \ref{t2} is satisfied. Then
\begin{equation*}
d(T^{n+1}(x_0),T^{n+1}(y_0))\leq (d(T^n(x_0),T^{n+1}(y_0))\cdot d(T^n(y_0),T^{n+1}(x_0)))^\lambda.
\end{equation*}
On the other hand,
\begin{multline*}
(d(T^n(x_0),T^{n+1}(y_0))\cdot d(T^n(y_0),T^{n+1}(x_0)))^\lambda\\
\leq (d(T^n(x_0),z)\cdot d(z,\acute{z})\cdot d(\acute{z},T^{n+1}(y_0))\cdot d(T^n(y_0),z)\cdot d(z,\acute{z})\cdot d(\acute{z},T^{n+1}(x_0)))^\lambda\\
\leq  d(z,\acute{z})^{2\lambda}\cdot\epsilon^{4\lambda}
\end{multline*}
Since
\begin{equation*}\epsilon< d(z,\acute{z})^{\frac{1}{2}\cdot\frac{1-2\lambda}{1+2\lambda}}\Rightarrow d(z,\acute{z})^{2\lambda}\cdot\epsilon^{4\lambda}<\frac{d(z,\acute{z})}{\epsilon^2},
\end{equation*}
Thus
\begin{eqnarray*}
(d(T^n(x_0),T^{n+1}(y_0))\cdot d(T^n(y_0),T^{n+1}(x_0)))^\lambda<\frac{d(z,\acute{z})}{\epsilon^2}\leq d(T^{n+1}(x_0),T^{n+1}(y_0)).
\end{eqnarray*}
Which is again a contradiction. Hence $z=\acute{z}$, which completes the proof.
\end{proof}

The following result extent Theorem\ref{t2}.

\begin{theorem}\label{t23}
Let $(M,d)$ be a  multiplicative metric space and $T : M \rightarrow M $ a continuous function such that for each pair of distinct points $x , y \in M$ at least one of the following conditions is satisfied:

$ (I)\  \ d(Tx,Ty)< (d(x,y))$

$(II)\  \ d(Tx,Ty)< \big(d(x,Tx).d(y,Ty)\big)^\frac{1}{2}$

$(III)\  \ d(Tx,Ty)< \big (d(x,Ty).d(y,Tx)\big)^\frac{1}{2}$\\
If for some $x_0\in M$, the sequence  $\{T^n(x_0)\}^\infty_{n=0}$ has a limit point $z$ in $M$, then $z$  is a unique fixed point $T$.
\end{theorem}
\begin{proof}
Consider subsequence $\{T^{n_i}(x_0)\}^\infty_{i=0}$ of $\{T^n(x_0)\}^\infty_{n=0}$ such that
\begin{equation*}
\lim_{i\to\infty}T^{n_i}(x_0)=z.
\end{equation*}
Then due to the continuity of $T$ we have
\begin{equation*}
\lim_{i\to\infty}T^{n_{i+1}}(x_0)=Tz \  \ \hbox{and}   \ \lim_{i\to\infty}T^{n_{i+2}}(x_0)=T^2z .
\end{equation*}
Next we show that for each nonnegative integer $n$ either $T^n(x_0)=T^{n+1}(x_0)$

or $$d(T^{n+1}(x_0),T^{n+2}(x_0))<d(T^n(x_0),T^{n+1}(x_0)).$$
Suppose, that there exist  some nonnegative integer $m$, such that
$T^m(x_0)\neq T^{m+1}(x_0)$
\hbox{and}
\begin{equation}\label{p}
  d(T^m(x_0),T^{m+1}(x_0))\leq d(T^{m+1}(x_0),T^{m+2}(x_0)).
 \end{equation}
Clearly  for distinct points $x=T^m(x_0)$ and $T^{m+1}(x_0)=y$,    inequality(\ref{p}) violates  each of the following three expression   representing conditions $(I), (II) and (III)$ of Theorem \ref{t23} respectively.
$$ \hspace{-5cm}d(T^{m+1}(x_0),T^{m+2}(x_0))<d(T^m(x_0),T^{m+1}(x_0)).$$
 $$d(T^{m+1}(x_0),T^{m+2}(x_0))<\big(d(T^m(x_0),T^{m+1}(x_0))\cdot d(T^{m+1}(x_0),T^{m+2}(x_0))\big)^\frac{1}{2}$$
\begin{eqnarray*}
d(T^{m+1}(x_0),T^{m+2}(x_0))&<&\big(d(T^m(x_0),T^{m+2}(x_0))\cdot 1\big)^\frac{1}{2}
\\&\leq& \big(d(T^m(x_0),T^{m+1}(x_0))\cdot d(T^{m+1}(x_0),T^{m+2}(x_0))\big)^\frac{1}{2}
\end{eqnarray*}

Hence for $T^m(x_0)\neq T^{m+1}(x_0)$, we have
\begin{equation*}\label{1}
  d(T^{m+1}(x_0),T^{m+2}(x_0)) <  d(T^m(x_0),T^{m+1}(x_0)).
 \end{equation*}

Therefore $\{ d(T^n(x_0),T^{n+1}(x_0)\}^\infty_{n=0}$ is monotonically decreasing and hence  convergent.\\Moreover
\begin{eqnarray*}
d(z,T(z))&=&\lim_{i\to\infty}d(T^{n_i}(x_0),T^{n_{i+1}}(x_0))=\lim_{i\to\infty}d(T^n(x_0),T^{n+1}(x_0))
\\&=&\lim_{i\to\infty}d(T^{n_{i+1}}(x_0),T^{n_{i+2}}(x_0))=d(T(z),T^2(z)).
\end{eqnarray*}
Now due to the uniqueness of limit point, it follows that $T(z)=z$.

For the uniqueness of fixed point, suppose $\acute{z}\neq z$ is another fixed point of $T$. Then it can be easily shown, as in the proof of Theorem \ref{t2}, that none of the three conditions of the theorem is satisfied for these distinct points. It completes the proof.
\end{proof}

One   can deduce the following  corollaries from  Theorem \ref{t23}
\begin{cor}\label{c1'} Let $M$ be a  multiplicative metric space, and let $ T: M\to M $ be a continuous  function such that for each pair of distinct points of $M $ condition $(I)$ of Theorem \ref{t23} is satisfied, and  if for some $x_0\in M$, the sequence  $\{T^n(x_0)\}^\infty_{n=0}$ has a limit point $z$ in $M$. Then $T$ has a unique fixed point.
\end{cor}
\begin{cor}\label{c2'} Let $M$ be a  multiplicative metric space, and let  $ T: M\to M $ be a continuous function such that for each couple of distinct points of $M $ condition $(II)$ of Theorem \ref{t23} is satisfiedand  if for some $x_0\in M$, the sequence  $\{T^n(x_0)\}^\infty_{n=0}$ has a limit point $z$ in $M$. Then $T$ has a unique fixed point.
\end{cor}
\begin{cor}\label{c3'} Let $M$ be a  multiplicative metric space, and let  $ T: M\to M $ be a continuous function such that for each couple of distinct points of $M $ condition $(III)$ of Theorem \ref{t23} is satisfiedand  if for some $x_0\in M$, the sequence  $\{T^n(x_0)\}^\infty_{n=0}$ has a limit point $z$ in $M$. Then $T$ has a unique fixed point.
 \end{cor}

 The next theorem is an improvement of Theorem \ref{t23}. The proof is simple therefore we omit it.

\begin{theorem}\label{t24}
 Let $M$ be a multiplicative  metric space and  $T:M\to M$ a function such that for each couple of different points in  $M$, at least one  of the conditions  $ (I), \  (II), \    (III)$ of Theorem \ref{t23} is satisfied.  If for some  $x_0\in M$, the  sequence  $\{T^nx_0 \}_{n=0}^\infty $ has a limit point $z\in M$ and $T$ is  continuous  at  $z$  and  at  $T(z)$,  then  $z$  is  a  unique fixed  point  of $T$.
\end{theorem}
The  condition of continuity for mapping is relaxed in the following theorem
\begin{theorem}\label{t25}
 Let $M$ be a multiplicative metric space and  $T:M\to M$ a function such that for each couple of different points in  $M$, at least one  of the conditions  $ (I),\    (II), \     (III)$ of Theorem \ref{t23} is satisfied.  If for some  $x_0\in M$, the  sequence  $\{T^nx_0 \}_{n=0}^\infty $ converge to a  point $z\in M$  then  $z$  is  a  unique fixed  point  of $T$.
\end{theorem}

\begin{proof}
The uniqueness part can easily be derived  similarly to Theoren\ref{c1}. To show the existence of fixed point, again following the lines in the proof Theorem \ref{c1}, because for the conditions 1),2) and 3) of Theorem \ref{t2} to be true, the use of the constants $\xi ,\eta$ and $\lambda$ is not essential. One have to show only the validity of at least one  of the three conditions of Theorem \ref{t23}.
\end{proof}

\begin{rem}
Theorems \ref{t2},\ \ref{t3'}, \ \ref{t22}, \  \ref{t23}, \  \ref{t24} and \ref{t25} carry fixed point  results  of Zamfirescu \cite{A1} in  metric spaces to multiplicative metric spaces.
\end{rem}

\begin{rem}
 Corollaries \ref{c1'}, \ref{c2'} and \ref{c3'}  carry some fixed point theorems   of  Edelstein\cite{eldesten},  Singh\cite{singh}  and  Chatterjea\cite{chettar} in metric spaces  to multiplicative metric spaces respectively, and  extent some results of "{O}zavsar and Cevikel\cite{A} in multiplicative metric spaces.
 \end{rem}


\begin{theorem}\label{th3}
  Let (M,d) be a complete multiplicative metric space and let \\$T : M \rightarrow M $ be a function such that for all $u , v \in M$
  \begin{equation}\label{eq1}
     d(Tu,Tv)\leq \frac{(d(u,Tu).d(v,Tv))^\frac{1}{2}}{\varphi(d(u,Tu),d(v,Tv))}.
  \end{equation}
   Where $\varphi:[1,\infty)\times[1,\infty)\rightarrow [1,\infty)$ is continuous and $\varphi (x,y)=1 $ if and only if $x=y=1$.
   Then $T$ has a unique fixed point in $M$.
\end{theorem}
\begin{proof}
    Let $u_0\in M $ be an arbitrary point. Clearly by induction a sequence $\{ u_n\}$ can be constructed such that
\begin{equation}\label{eq2}
      u_{n+1}=Tu_n=T^{n+1}u_0\hspace{1cm} \hbox{for} \hspace{1mm} \hbox{all} \hspace{1mm} n\geq 0.
\end{equation}
If for some natural number $n\ , \ u_n=u_{n+1}$ then clearly $u_n$ is fixed point of $T$, which complete the proof.

 Suppose $u_n\neq u_{n+1}$ \hbox{for all} $n\in N$. To discuss this case  we executing the following steps:

 \textbf{Step 1.} We claim that
 \begin{equation}\label{eq3} \lim_{n\to \infty}d(u_n,u_{n+1})= 1.\end{equation}
Putting  $u=u_n$ and $v= u_{n-1}$ in \eqref{eq1}, and using the  properties of $\varphi$ function  we have
\begin{eqnarray}\label{eq4}
 \nonumber d(u_{n+1},u_n)
 \nonumber &\leq & \frac{\big(d(u_n,Tu_n).d(u_{n-1},Tu_{n-1})\big)^\frac{1}{2}}{\varphi \big(d(u_n,Tu_n),d(u_{n-1},Tu_{n-1})\big)} \\ &= &  \frac{\big(d(u_n,u_{n+1}).d(u_{n-1},u_n)\big)^\frac{1}{2}}{\varphi \big(d(u_n,u_{n+1}),d(u_{n-1},u_n)\big)}\\\nonumber &\leq &\big(d(u_n,u_{n+1}).d(u_{n-1},u_n)\big)^\frac{1}{2}\hspace{1cm} \\\nonumber
\end{eqnarray}
 which implies that $$ d(u_{n+1},u_n) \leq d(u_n, u_{n-1})\hspace{5mm} \text{for all}\hspace{1mm} n\geq 1.$$
Therefore $\{d(u_n, u_{n+1})\}$ is a sequence which is monotonically non-increasing and bounded below by 1. So, there must be some $r\geq1$ such that
\begin{equation}
\nonumber \lim_{n\to \infty}d(u_n, u_{n+1})=r.
\end{equation}
 Letting $n\to\infty$ in \eqref{eq4} and using the continuity of $\varphi$, we get $ r\leq \frac {(r.r)^\frac{1}{2}}{\varphi(r,r)} ,$
 which implies that
$ \varphi(r,r) \leq1,  \   \hbox{ which shoes that } $ r=1$.$

  Thus \eqref{eq3} is proved.

  \textbf{Step 2.} Next we shall prove that
 \begin{equation}\label{eq5}
 \lim_{n\to \infty}d(u_n,u_{n+2})= 1.
 \end{equation}
 From  \eqref{eq1}
  we have.\begin{multline}\nonumber d(u_{n+2},u_n)\leq \frac{\big(d(u_{n+1},Tu_{n+1}).d(u_{n-1},Tu_{n-1})\big)^\frac{1}{2}}{\varphi\big(d(u_{n+1},Tu_{n+1}),d(u_{n-1},Tu_{n-1})\big)}\\\nonumber = \frac{\big(d(u_{n+1},u_{n+2}).d(u_{n-1},u_n)\big)^\frac{1}{2}}{\varphi\big(d(u_{n+1},u_{n+2}),d(u_{n-1},u_n)\big)}\nonumber \leq \big(d(u_{n+1},u_{n+2}).d(u_{n-1},u_n)\big)^\frac{1}{2}.\hspace{1cm} \end{multline}
 Taking limit $n\to\infty$  and  using   \eqref{eq3} we get
 $$\lim_{n\to\infty}d(u_{n+2},u_n) = 1.$$

 So \eqref{eq5} is proved.

  \textbf{Step 3.} In this step we show  that $T$ has a periodic point. Suppose by the way of contradiction that $T$ has no periodic point, then $\{u_n\}$ is sequence of distinct points,  that is $u_n\neq u_m \text{ for all } n\neq m.$   In this case $\{u_n\}$ is a multiplicative Cauchy sequence. If we suppose that $\{u_n\}$ is not  multiplicative Cauchy  then there will exists  some $\epsilon> 1$ such that for an integer $q$ there exist integers $m(q)>n(q)>q$ such that
  \begin{eqnarray}\label{eq6}d(u_{m(q)},u_{n(q))}>\epsilon. \end{eqnarray}

  For every positive integer $ q \text{, let }m(q) \text{ be the least positive integer exceeding } n(q )$ and  satisfying \eqref{eq6} then  \begin{eqnarray}\label{eq7}d(u_{n(q)},u_{m(q)-1)}\leq\epsilon.\end{eqnarray}
From  \eqref{eq6},\eqref{eq7} and multiplicative triangular inequality we have.
\begin{eqnarray} \nonumber\epsilon< d(u_{m(q)},u_{n(q))}&\leq& d(u_{m(q)},u_{m(q)-2}).d(u_{m(q)-2},u_{m(q)-1}).d(u_{m(q)-1},u_{n(q)})\\\nonumber &\leq& d(u_{m(q)},u_{m(q)-2}).d(u_{m(q)-2},u_{m(q)-1}).\epsilon
\end{eqnarray}
Taking limit as $n\rightarrow\infty$ and using \eqref{eq3} and \eqref{eq5} one can get
\begin{equation}\label{eq8}\nonumber \epsilon< \lim_{q\to \infty}d(u_{m(q)},u_{n(q))}\leq\epsilon.
\end{equation}
Thus
\begin{equation}\label{eq9}
  \lim_{q\to \infty}d(u_{m(q)},u_{n(q))} =\epsilon.
\end{equation}
Putting  $ u= u_{m(q)-1} \text{ and } v= u_{n(q)-1}$ in equation \eqref{eq1}, we get
\begin{equation}\label{eq9'}
  d(u_{m(q)},u_{n(q))}\leq \frac{\big(d(u_{m(q)-1},u_{m(q)}).d(u_{n(q)-1},u_{n(q)})\big)^\frac{1}{2}}{\varphi\big(d(u_{m(q)-1},u_{m(q)}),d(u_{n(q)-1},u_{n(q)})\big)}.
\end{equation}
 Letting  $ q \to \infty$ in \eqref{eq9'}and using \eqref{eq3}, \eqref{eq8} and continuity of $\varphi$ we have
 \begin{eqnarray*}
 \epsilon \leq \frac{1}{\varphi(1,1)}= 1.
   \end{eqnarray*}
 Which is a contradiction. Hence $\{u_n\}$ is a multiplicative Cauchy sequence. Since $(M,d)$ is a complete multiplicative metric space , so there will be some $w\in M $ such that $u_n\to w $. Substituting $u=u_n \text{ and } v=w $ in equation \eqref{eq1} we get
 \begin{equation}\label{eq9}
 d(u_{n+1},Tw) \leq \frac{\big(d(u_n,u_{n+1}).d(w,Tw)\big)^\frac{1}{2}}{\varphi\big(d(u_n,u_{n+1}),d(w,Tw)\big)}
 \end{equation}
 Since $\varphi(x,y)\geq 1$, therefore
 \begin{equation*} d(u_{n+1},Tw)\leq \big(d(u_n,u_{n+1}).d(w,Tw)\big)^\frac{1}{2}.
 \end{equation*}Letting $n\to \infty$ and using \eqref{eq3} we obtain
 \begin{equation}\label{eq10}\lim_{n\to \infty}d(u_{n+1},Tw)\leq \{d(w,Tw)\}^\frac{1}{2}.
 \end{equation}

 Now we shall find a contradiction to the assumption that $T$ has no periodic point in each of the following two cases.

  $\textbf{Case 1.}$  If $ u_n\neq w \text{ and } u_n\neq Tw \text{ for all } n\geq 2, $ Then using multiplicative triangular inequality we have
 \begin{equation*}\nonumber  d(w,Tw)\leq d(w,u_n)\cdot d(u_n,u_{n+1})\cdot d(u_{n+1},Tw).
 \end{equation*} Letting  $n\to \infty $ and using \eqref{eq3} we have

 \begin{eqnarray}\label{eq11} d(w,Tw)&\leq&\lim_{n\to \infty}d(u_{n+1},Tw).
 \end{eqnarray}
  From \eqref{eq10} and \eqref{eq11} it follows that
 \begin{equation*}
 \ d(w,Tw)\leq\{d(w,Tw)\}^\frac{1}{2},
 \end{equation*}
   which is possible only if $d(w,Tw)=1$, which shoes that $ Tw= w $that is, $w$ is fixed point of $T$, so $w$ is periodic point of $T$. Which is a contradiction to the assumption that $T$ has no periodic point.

   $\textbf{Case 2.}$ $ \text{ If for some } p\geq 2, u_p=w \text{ or } u_p=Tw .$ Since $T $ has no periodic point, therefore $w\neq u_0$ because otherwise $u_p=w =u_0\Rightarrow T^pu_0=u_0$ that is $u_0$ is periodic point of $T$.
   Also if $u_p=Tw \text{ and } w= u_0 \text{ then } Tu_0 =Tw= u_p=T^pu_0=T^{p-1}(Tu_0)$, i.e, $Tu_0$ is periodic point of $T$.
    Hence in either case we have a contradiction to the fact that $T$ has no periodic point. Now for all $n\geq 1$, we have
   \begin{equation*} d(T^nw,w)=d(T^nu_p,w)=d(u_{n+p},w)\  \ \text{   or}
   \end{equation*}
   \begin{equation*}d(T^nw,w)=d(T^{n-1}Tw,w)=d(T^{n-1}u_p,w)=d(u_{n+p-1},w).
   \end{equation*} \\In the above two identities the integer $ p\geq 2$ is fixed, so $\{ u_{n+p}\}$ and $\{ u_{n+p-1}\}$ are subsequence of $\{ u_n\},$ and as $\{u_n\}$ is multiplicative sequence converging to $w$ in multiplicative metric space, so these two subsequences are also multiplicative convergent to the same unique limit $w$, i.e
   \begin{eqnarray}\nonumber \lim_{n\to\infty}d(u_{n+p},w)= \lim_{n\to\infty}d(u_{n+p-1},w)=1
   \end{eqnarray} Thus
   \begin{eqnarray}\label{eq12}\lim_{n\to\infty}d(T^nw,w)=1.
   \end{eqnarray} Also, since $\{T^{n+1}w\}$  is subsequence of  $\{T^nw\}$, therefore \begin{eqnarray}\label{eq13}\lim_{n\to\infty}d(T^{n+1}w,w)=1.
   \end{eqnarray}
   According to supposition  $T$ has no periodic point, therefor $T^rw\neq T^sw $ for any $r,s\in N$ where $r\neq s$. Using multiplicative reverse triangular inequality we have
   \begin{eqnarray*}\frac{1}{d(T^{n+1}w,w)}\leq\frac{d(T^{n+1}w,Tw)}{d(w,Tw)}\leq d(T^{n+1}w,w).
   \end{eqnarray*}
   Using \eqref{eq13} we have
   \begin{eqnarray*}\label{eq14}\nonumber\lim_{n\to \infty} \frac{1}{d(T^{n+1}w,w)}&\leq&\lim_{n\to \infty}\frac{d(T^{n+1}w,Tw)}{d(w,Tw)}\leq \lim_{n\to \infty}d(T^{n+1}w,w) \\\nonumber 1&\leq&\lim_{n\to \infty}\frac{d(T^{n+1}w,Tw)}{d(w,Tw)}\leq 1
   \end{eqnarray*}
   Thus
   \begin{equation}\label{eq13}
     \lim_{n\to \infty}d(T^{n+1}w,Tw)= d(w,Tw).
   \end{equation}
  Now by \eqref{eq1}
   \begin{eqnarray*} d(T^{n+1}w,Tw)\leq \frac{\{d(T^nw,T^{n+1}w)\cdot d(w,Tw)\}^\frac{1}{2}}{\varphi\{d(T^nw,T^{n+1}w),d(w,Tw)\}}
   \end{eqnarray*}
   Now letting $n\to \infty$ and using \eqref{eq3} and \eqref{eq14} we have,
   \begin{equation*}  d(w,Tw)\leq \frac{\{d(w,Tw)\}^\frac{1}{2}}{\varphi\{1,d(w,Tw)\}}\leq \{d(w,Tw)\}^\frac{1}{2}\hspace{2cm}
   \end{equation*}
        Which is true only if $d(w,Tw)=1$, and hence  $Tw =w$
  That is $w$ is periodic point of $T$. It contradicts the fact that $T$ has no periodic point.

  Thus   $T$ has a periodic point. It means there exists a point $w\in M $ such that $T^pw=w $ for some intger $p\geq 1.$

   \textbf{Step 4.} Here we will prove that $T$ has fixed point. Sine $T$ has  periodic point therefore  $T^pw=w $ for some intger $p\geq 1$ and $w\in M$. Now if $p=1$ then $Tw=w$ that is $w$ is fixed point of $T$.

     Assume that $p>1$. We will show that $a=T^{p-1}w$ is fixed point of $T$. Let us suppose by the way of contradiction that $T^{p-1}w\neq T^pw.$
     Then clearly  $ d(T^{p-1}w,T^pw)>1$ and   $\varphi(d(T^{p-1}w,T^pw),d(T^{p-1}w,T^pw))>1$.
     \\ By using  \eqref{eq1} we get
   \begin{multline}\label{eq15} \nonumber d(w,Tw)=d(T^pw,T^{p+1}w)\nonumber=
     d(T(T^{p-1}w),T(T^pw))\\\nonumber \leq \frac{\big(d(T^{p-1}w,T^pw).d(T^pw,T(T^pw))\big)^\frac{1}{2}}{\varphi\big(d(T^{p-1}w,T^pw),d(T^pw,T(T^pw))\big)}\nonumber < \big(d(T^{p-1}w,T^pw).d(T^pw,T(T^pw))\big)^\frac{1}{2}.\\
     =\big(d(T^{p-1}w,T^pw).d(w,Tw)\big)^\frac{1}{2}
  \end{multline}
  Since $d(w,Tw)< d(w,Tw)^\frac{1}{2}$ is not possible therefore we have
  \begin{equation}\label{eq15}
     d(w,Tw)<d(T^{p-1}w,T^pw)
  \end{equation}
  Again using \eqref{eq1} we get
  \begin{multline}\label{eq16}\nonumber d(T^{p-1}w,T^pw)= d(T(T^{p-2}w),T(T^{p-1}w))\\\nonumber \leq \frac{\big(d(T^{p-2}w,T^{p-1}w)\cdot d(T^{p-1}w,T(T^{p-1}w))\big)^\frac{1}{2}}{\varphi\big(d(T^{p-2}w,T^{p-1}w),d(T^{p-1}w,T(T^pw))\big)}\\\nonumber < \big(d(T^{p-2}w,T^{p-1}w)\cdot d(T^{p-1}w,T(T^{p-1}w))\big)^\frac{1}{2}\\\nonumber = \big(d(T^{p-2}w,T^{p-1}w)\cdot d(T^{p-1}w, T^pw)\big)^\frac{1}{2}
  \end{multline}
  Finally we have
  \begin{equation}\label{eq16}
 d(T^{p-1}w,T^pw)< d(T^{p-2}w,T^{p-1}w).
  \end{equation}
  Continuing in this way as in \eqref{eq15} and \eqref{eq16} we get
  \begin{equation*}d(w,Tw)<d(T^{p-1}w,T^pw)< d(T^{p-2}w,T^{p-1}w)< d(T^{p-3}w,T^{p-2}w)<\cdots < d(w,Tw).
  \end{equation*}Which is a contradiction. Hence we conclude that $w=T^{p-1}w$ is a fixed point of $T$.

  \textbf{Step 5.}  In this step we show the uniqueness of the fixed point.
  Suppose that fixed point of $T$ is not unique. Rather $a,b\in M$  where $a\neq b$ are two distinct fixed points of $T$ , i.e $Ta=a$  and  $Tb=b $. Then using \eqref{eq1} we have
  \begin{equation*}
  d(a,b)\leq \frac{\big(d(a,Ta)\cdot d(b,Tb)\big)^\frac{1}{2}}{\varphi\big(d(a,Ta),d(b,Tb)\big)}
  =\frac{(\big(d(a,a)\cdot d(b,b)\big))^\frac{1}{2}}{\varphi\big(d(a,a),d(b,b)\big)}=1
\end{equation*}
which shoes that  $a=b$

 Hence $T$ has unique fixed point. This completes the proof.
\end{proof}
The above theorem produces the following corollaries:
\begin{cor}\label{cc7}
 Let $(M,d)$ be a complete multiplicative metric space. Let $T:M\to M$ be such that for all $u,v\in M$ there exists $q\in[0,1)$ and $d(Tu,Tv)\leq(d(u,Tu)\cdot d(v,Tv))^\frac{q}{2}.$ Then $T$ has unique fixed point.
\end{cor}
\begin{proof} Taking $\varphi(x,y)= (x\cdot y)^\frac{1-q}{2}  $ in Theorem \ref{th3} we have
\begin{eqnarray*}
       d(Tu,Tv)&\leq& \frac{\big(d(u,Tu).d(v,Tv)\big)^\frac{1}{2}}{\varphi\big(d(u,Tu),d(v,Tv)\big)}\\&=& \frac{\big(d(u,Tu)\cdot d(v,Tv)\big)^\frac{1}{2}}{\big(d(u,Tu)\cdot d(v,Tv)\big)^\frac{1-q}{2}}\\&=&\big(d(u,Tu)\cdot d(v,Tv)\big)^\frac{q}{2}.
 \end{eqnarray*}
    Hence Theorem \ref{th3} completes the proof.
 \end{proof}
 \begin{cor}\label{cc8}
 Let $(M,d)$ be a complete multiplicative metric space and let $T : M \rightarrow M $ be a function such that for all $u , v \in M$
\begin{equation}\label{eq18}
  d(Tu,Tv)\leq \frac{\big(d(u,Tu).d(v,Tv)\big)^\frac{1}{2}}{\psi\big(d(u,Tu)\cdot d(v,Tv)\big)^\frac{1}{2}}.
\end{equation}
Where $\psi:[1,\infty)\rightarrow [1,\infty)$ is continuous and $\psi^{-1}(1)=1 $
Then $T$ has a unique fixed point in $M$.
 \end{cor}
 \begin{proof}
 Here $\varphi(s,t)=\psi(s\cdot t)^\frac{1}{2}$. Clearly $\varphi(s,t)$ is continuous, and $\varphi(s,t)=1\Leftrightarrow \psi(s\cdot t)^\frac{1}{2}=1$ $\Leftrightarrow (s\cdot t)^\frac{1}{2}=\psi^{-1}(1)=1\Leftrightarrow s=t=1.$ That is $\varphi(s,t)$ satisfies hypotheses of the Theorem \ref{th3}. Hence proof follows from Theorem \ref{th3}.
 \end{proof}

  \begin{rem}
  Theorem \ref{th3}, Corollaries  \ref{cc7} and \ref{cc8} carry fixed point  results of Aydi \textit{et al} \cite{2} in complete Hausdorff generalized metric spaces to multiplicative metric spaces.
\end{rem}

 We conclude with the following supporting examples.

\begin{ex} Consider multiplicative metric space $(R^2,d)$, with multiplicative metric $d$ defined as $$d((x_1,y_1),(x_2,y_2))=2^{\sqrt{(x_1-x_2)^2+(y_1-y_2)^2}}$$.
Then the mapping $T:R^2\to R^2$ defined as $T((x,y))=\frac{2}{3}(x,y)$ satisfies the following multiplicative contraction:
$$d(T(x_1,y_1),T(x_2,y_2))= d((x_1,y_1),(x_2,y_2))^\xi\ \  \hbox{for all}\ \   (x_1,y_1), (x_2,y_2)\in R^2, $$ where  $\xi=\frac{2}{3}\in [0,1)$.That is condition (1) of Theorem \ref{t2} is satisfied. One can easily check that $T$ has unique fixed point $(0,0)\in R^2$.
\end{ex}
\begin{ex} Let $M=[0.1,1]$. Consider the multiplicative metric $d:M\times M\to R$ defined by
 $$d(x,y)= e^{|\frac{1}{x}-\frac{1}{y}|}.$$
Then $(M,d)$ is complete multiplicative metric space. Let $T: M\to M$ be a mapping given by
$T(x)= \frac{1}{2+x}.$
 Then the mapping holds the following multiplicative contraction conditions  (2) and (3) of Theorem \ref{t2}   respectfully
  $$d(T(x),T(y))\leq (d(x,T(x))\cdot d(y,T(y)))^\eta   \ \ for all\ \  x,y \in X$$
  $$d(T(x),T(y))\leq (d(x,T(y))\cdot d(y,T(x)))^\lambda \ \  for all\ \   x,y \in X$$
  with $\eta=\lambda=0.499 \in [0,\frac{1}{2})$. Obviously $T$ has unique fixed point $0.4142135624 \in X$
\end{ex}

\begin{ex} Let $M=[1,2]$. Consider the multiplicative metric $d:M\times M\to R$ defined by
$$d(x,y)= 2^{|x-y|}.$$
Then $(M,d)$ is complete multiplicative metric space. Let $T: M\to M$ be a mapping given by
$T(x)= \frac{1}{\sqrt{x}}.$ and define $\varphi$  as
 $$\varphi(x,y) = 2^{(|x-\frac{1}{\sqrt{x}}|+|y-\frac{1}{\sqrt{y}}|)\cdot 10^{-5}}.$$
  Then $T$ satisfies all the  condition of the Theorem \ref{th3}. i.e.
 \begin{equation*}d(Tx,Ty)\leq \frac{\{d(x,Tx).d(y,Ty)\}^\frac{1}{2}}{2^{\big(|x-\frac{1}{\sqrt{x}}|+|y-\frac{1}{\sqrt{y}}|\big)\cdot 10^{-5}}} \text{  for all } x,y \in M.
 \end{equation*}
  So  $T$ has unique fixed point which is  $1\in M.$
\end{ex}

{\bf Muhammad Sarwar}\\
Department of Mathematics
University of Malakand, Chakdara Dir(L), Khyber PakhtunKhwa, Pakistan.\\
Email: sarwar@uom.edu.pk, sarwarswati@gmail.com\\
{\bf Badshah-e-Rome}\\
Department of Mathematics
University of Malakand, Chakdara Dir(L), Khyber PakhtunKhwa, Pakistan.\\
Email: baadeshah$@$yahoo.com

\end{document}